\documentclass[a4paper,12pt,reqno]{article}
\usepackage[english]{babel}
\usepackage{amsmath,amsfonts,amssymb,amsthm,bbm}
\usepackage{graphics,epsfig} 
\usepackage{paralist,subfigure}
\usepackage{hyperref} 
\usepackage[active]{srcltx} 
\usepackage{color,soul} 
\usepackage[latin1]{inputenc}
\usepackage[OT1]{fontenc}

\newtheorem{theorem}{Theorem}[section]

\newtheorem{lemma}[theorem]{Lemma}
\newtheorem{proposition}[theorem]{Proposition}

\theoremstyle{definition}

\newtheorem{remark}[theorem]{Remark}

\newcommand{\baa}{\begin{array}}
\newcommand{\eaa}{\end{array}}
\newcommand{\ba}{\begin{eqnarray}}
\newcommand{\ea}{\end{eqnarray}}
\newcommand{\be}{\begin{equation}}
\newcommand{\ee}{\end{equation}}

\DeclareMathOperator\Tr{Tr}

\def\N{\mathbb{N}}

\def\R{\mathbb{R}}

\let\O=\Omega
\let\e=\varepsilon
\def\epsilon{\varepsilon}
\let\vp=\varphi

\let\t=\tilde
\let\ol=\overline
\let\ul=\underline
\let\.=\cdot
\let\0=\emptyset

\def\1{\mathbbm{1}}

\def\pf{principal eigenfunction}

\newenvironment{formula}[1]{\begin{equation}\label{#1}}{\end{equation}\noindent}

\def\Fi#1{\begin{formula}{#1}}
\def\Ff{\end{formula}\noindent}

\newcommand{\SE}{\setcounter{equation}{0} \section}

\def\lt{\lambda(\tau)}
\def\pt{\varphi_\tau}

\begin{document}

\title{{\bf Optimization of some eigenvalue problems with large drift}\thanks{The project leading to this publication has received funding from Excellence Initiative of Aix-Marseille University - A*MIDEX and Excellence Laboratory Archimedes LabEx (ANR-11-LABX-0033), French ``Investissements d'Avenir" programmes, from the ANR NONLOCAL project (ANR-14-CE25-0013) funded by the French National Research Agency, and from the European Research Council under the European Union's Seventh Framework Programme (FP/2007-2013) ERC Grant Agreement n.~321186~- ReaDi~- Reaction-Diffusion Equations, Propagation and Modelling.}}

\author{Fran{\c{c}}ois Hamel$^{\,\hbox{\small{a}}}$, Luca Rossi$^{\,\hbox{\small{b}}}$ and Emmanuel Russ$^{\,\hbox{\small{c}}}$\\
\\
\footnotesize{$^{\,\hbox{\small{a}}}$ Aix Marseille Univ, CNRS, Centrale Marseille, I2M, Marseille, France}\\
\footnotesize{$^{\,\hbox{\small{b}}}$ CNRS, EHESS, PSL Research University, CAMS, Paris, France}\\
\footnotesize{$^{\,\hbox{\small{c}}}$ Universit\'e Grenoble Alpes, Institut Fourier}\\
\footnotesize{100 rue des maths, BP 74, 38402 Saint-Martin d'H\`eres Cedex, France}}
\date{}

\maketitle

\begin{abstract}
This paper is concerned with eigenvalue problems for non-symmetric elliptic operators with large drifts in bounded domains under Dirichlet boundary conditions. We consider the minimal principal eigenvalue and the related principal eigenfunction in the class of drifts having a given, but large, pointwise upper bound. We show that, in the asymptotic limit of large drifts, the maximal points of the optimal principal eigenfunctions converge to the set of points maximizing the distance to the boundary of the domain. We also show the uniform asymptotic profile of these principal eigenfunctions and the direction of their gradients in neighborhoods of the boundary.
\end{abstract}


\SE{Introduction}\label{intro}

Throughout this paper, if $n\geq 1$, by ``domain'' of $\R^n$, we mean an open connected subset of $\R^n$. The set of all bounded domains of $\R^n$ with $C^2$ boundary will be denoted by ${\mathcal O}$. If $\Omega\in {\mathcal O}$ and $x\in \overline{\Omega}$, define $d(x):=d(x,\partial\Omega)$. For $\delta>0$, let
\begin{equation}\label{Odelta}
\Omega^{\delta}=\big\{x\in\O:\ d(x)<\delta\big\}
\end{equation}
be the open neighborhood of $\partial\O$ of width $\delta$, relatively to $\O$. The Euclidean norm in $\R^n$ will be denoted by $\left\vert \cdot \right\vert$. When $v:\Omega\rightarrow \R^n$ is a measurable vector field on $\Omega$, we say that $v\in L^{\infty}(\Omega,\R^n)$ if and only if $\left\Vert v\right\Vert_{\infty}:=\left\Vert \left\vert v\right\vert\right\Vert_{L^{\infty}(\O)}<+\infty$. If $A\subset \R^n$ is measurable, $\left\vert A\right\vert$ stands for the Lebesgue measure of $A$. Finally, we 
let $B_r(x)$ denote the open Euclidean ball of $\R^n$ with center $x\in\R^n$ and radius $r>0$, and we set $B_r=B_r(0)$. With a slight abuse of notation, we also note $B_0(x)=\{x\}$ and $B_0=\{0\}$.

Let $n\geq 1$ and $\O\in {\mathcal O}$. Consider a bounded measurable vector field $v:\O\rightarrow\R^n$. We are interested in the principal eigenvalue of the operator $-\Delta+v\cdot\nabla$ in $\O$ under Dirichlet boundary condition, which will be denoted by $\lambda_v$ in the sequel.\footnote{Note that $\lambda_v$ depends on $\Omega$ and $v$, but the dependence with respect to $\Omega$ will not be explicitly written down. This is harmless for the present paper, since the domain $\Omega$ will be given and fixed for all the issues we consider.} Recall (\cite[p.~49]{BNV}) that $\lambda_v$ is real-valued and that the eigenspace corresponding to $\lambda_v$ has dimension $1$. Moreover, if $\vp$ is ``the'' eigenfunction in $L^2(\O)$ corresponding to the eigenvalue $\lambda_v$, then (up to normalization) $\vp>0$ in $\O$. Thus, the function $\vp$ satisfies
\Fi{ev}
\begin{cases}
-\Delta\vp+v\.\nabla\vp=\lambda_v\vp & \text{in }\O,\vspace{3pt}\\
\vp>0 & \text{in }\O,\\
\vp=0  & \text{on }\partial\O.
\end{cases}
\Ff
By standard elliptic regularity, $\vp\in W^{2,p}(\Omega)$ for all $p\in [1,\infty)$, which entails that $\vp\in C^{1,\alpha}(\overline{\Omega})$ for all $\alpha\in (0,1)$. The strong maximum principle also ensures that $\lambda_v>0$.

The main results of the present paper arise from the study of optimization pro\-blems for $\lambda_v$. An archetypical optimization problem is the following question: given any fixed $m>0$ and $\tau\geq 0$, among all the domains $\O\in {\mathcal O}$ with $\left\vert \Omega\right\vert=m$ and all the vector fields $v\in L^{\infty}(\Omega,\R^n)$ with $\left\Vert v\right\Vert_{\infty}\leq \tau$, is the infimum of $\lambda_v$ reached for some $\Omega$ and some $v$ ? When $\tau=0$, this amounts to minimizing the principal eigenvalue of~$-\Delta$ in $\Omega$ under Dirichlet boundary condition, and it is a well-known fact (\cite{Faber,Krahndim2,Krahndimn}) that the infimum is reached if and only if $\Omega$ is an Euclidean ball. For $\tau>0$, it was proved by the first and the third author (\cite[Theorem 2.9]{HNR}) that $\lambda_v$ reaches its infimum if and only if, up to translation, $\Omega$ is an Euclidean ball centered at $0$ and $v(x)=\tau\frac{x}{\left\vert x\right\vert}$. As a matter of fact, the first step of the proof of~\cite[Theorem 2.9]{HNR} given in \cite{HNRcras,HNRpreprint} is the solution of an optimization problem for $\lambda_v$ when the domain $\Omega$ is fixed and the vector field $v$ varies under the constraint $\left\Vert v\right\Vert_{\infty}\leq \tau$. More precisely, given $\Omega\in\mathcal{O}$, define, for $\tau\geq 0$,
\Fi{lt}
\lambda(\tau):=\inf\left\{\lambda_v :\ \left\Vert v\right\Vert_{\infty}\leq \tau\right\},
\Ff
where the infimum is taken over all the vector fields $v\!\in\! L^{\infty}(\Omega,\R^n)$ such that $\left\Vert v\right\Vert_{\infty}\!\leq\!\tau$. Then (\cite[Theorem 6.6]{HNR}), there exists a unique vector field $v_{\tau}\in L^{\infty}(\Omega,\R^n)$ with $\left\Vert v_{\tau}\right\Vert_{\infty}\leq \tau$ such that $\lambda(\tau)=\lambda_{v_{\tau}}$. Moreover, $\left\vert v_{\tau}(x)\right\vert=\tau$ for a.e. (almost every) $x\in \Omega$, and, if $\varphi_{\tau}$ is the corresponding eigenfunction of~\eqref{ev} with $v=v_\tau$, one has $|\nabla \varphi_{\tau}(x)|>0$ for a.e. $x\in\Omega$ and
\Fi{vmin}
v_{\tau}=-\tau\,\frac{\nabla\pt}{|\nabla\pt|}\qquad\text{a.e. in }\O.
\Ff
Thus, $\varphi_{\tau}$ solves
\Fi{evm}
\begin{cases}
-\Delta\varphi_{\tau}-\tau\left\vert \nabla\varphi_{\tau}\right\vert=\lambda(\tau)\varphi_{\tau} & \text{in }\O,\\
\pt>0 & \text{in }\O,\\
\pt=0  & \text{on }\partial\O.
\end{cases}
\Ff
This entails that, for all $\alpha\in (0,1)$, $\Delta\vp_{\tau}\in C^{0,\alpha}(\overline{\Omega})$, so that $\vp_{\tau}\in C^{2,\alpha}_{loc}(\Omega)$ (if $\Omega$ is assumed to be of class $C^{2,\alpha}$ for some $\alpha\in(0,1)$, the function $\pt$ would then be of class $C^{2,\alpha}(\overline{\O})$). We normalize $\vp_{\tau}>0$ by setting $\max_{\overline{\O}} \vp_{\tau}=1$. The function $\pt$ is then the unique solution to \eqref{evm} 
satisfying this normalization.
 Notice also that the equality $\lambda(\tau)=\lambda_{v_\tau}$ immediately yields
$$\lambda(\tau)>0.$$

When $\Omega=B_R$ for some $R>0$, then $\vp_{\tau}$ is radially decreasing and $v_{\tau}(x)=\tau\frac{x}{\left\vert x\right\vert}$ for all $x\in \Omega\setminus \left\{0\right\}$ (\cite[Theorem 6.8]{HNR}). Moreover (\cite[Lemma 7.2]{HNR}), $\ln\lambda(\tau)\sim -\tau R$ when $\tau\rightarrow +\infty$, and, when $n=1$, $\lambda(\tau)\sim \tau^2e^{-\tau R}$ when $\tau\rightarrow +\infty$. These asymptotics were also proved in \cite{friedman} by probabilistic arguments for more general elliptic operators of second order with $C^1$ coefficients (note that $\frac{x}{\left\vert x\right\vert}$ is not $C^1$ at $0$, so that the results of~\cite{friedman} do not exactly fall into the scope of the problems dealt with in the present paper).

This paper is chiefly devoted to the study of the asymptotic behavior of $\varphi_{\tau}$ as $\tau\!\rightarrow\!+\infty$ for {\it any} domain $\Omega\in\mathcal{O}$. Our first main result deals with the points where $\varphi_{\tau}$ reaches its maximum (recall that $\vp_{\tau}$ is continuous in $\overline{\O}$). When $\O$ is ball, as recalled before, this maximum is reached at the center of $\O$, {\it i.e.} the point in $\overline{\O}$ where $d$ reaches its maximum. In the general case, we establish:

\begin{theorem}\label{thmaxima}
Let $\Omega\in {\mathcal O}$. For all $\tau\geq 0$, let $x_{\tau}$ be a point in $\O$ where $\pt$ reaches its maximum in $\overline{\O}$. Then
$$d(x_\tau)\to \max_{\ol\O}d\quad\text{as }\tau\to+\infty.$$
\end{theorem}  

We then prove that $\vp_{\tau}$ converges to $1$ as $\tau\rightarrow +\infty$ locally uniformly in $\Omega$, and give the precise asymptotic profile near $\partial\Omega$:

\begin{theorem}\label{thconv}
Let $\Omega\in {\mathcal O}$. There holds
\begin{equation}\label{profile}
\frac{\vp_\tau(x)}{1-e^{-\tau d(x)}}\to1\quad\text{as }\tau\to+\infty,
\end{equation}
uniformly with respect to $x\in\O$.
\end{theorem}

Finally, we also describe the behavior of $\nabla\vp_{\tau}$ when $\tau\rightarrow +\infty$. In the case where $\Omega$ is a ball $B_R$, since $\vp_{\tau}$ is radial, $\frac{\nabla\vp_{\tau}(x)}{\left\vert \nabla \vp_{\tau}(x)\right\vert}=-\frac{x}{\left\vert x\right\vert}=\nabla d(x)$ for all $x\in \Omega\backslash\{0\}$. For a general domain $\Omega\in {\mathcal O}$, we compare $\frac{\nabla\vp_{\tau}}{\left\vert \nabla \vp_{\tau}\right\vert}$ with $\nabla d$ near $\partial\Omega$ when $\tau\rightarrow +\infty$. More precisely:

\begin{theorem}\label{thnabla}
Let $\Omega\in {\mathcal O}$. For any $M>0$, there holds
$$\min_{\ol{\O^{M/\tau}}}|\nabla\pt|\gtrsim\tau\ \ \hbox{ and }\ \ \max_{\ol{\O^{M/\tau}}}\,
\left|\frac{\nabla\pt}{|\nabla\pt|}-\nabla d\right|\to0\quad\text{as }\tau\to+\infty.\footnote{The notation $\min_{\ol{\O^{M/\tau}}}|\nabla\pt|\gtrsim\tau$ as $\tau\to+\infty$ means that there exist $C>0$ and $\tau_0>0$ such that $\min_{\ol{\O^{M/\tau}}}|\nabla\pt|\ge C\tau$ for all $\tau\ge\tau_0$.}$$
\end{theorem}

We conjecture that, for a general domain $\O\in\mathcal{O}$, the vector fields $\nabla\pt/|\nabla\pt|$ 
converge to $\nabla d$ a.e.~in $\O$ as $\tau\to+\infty$, that is, the minimizing vector fields $v_{\tau}$ given by~\eqref{vmin} are asymptotically proportional to the opposite of the gradient of the distance to $\partial\O$. That problem is still open for general domains $\O$. However, the property holds when $\O$ is an annulus, say with center $0$, that is, $\O=B_R\setminus\overline{B_r}$ for some $0<r<R$.

\begin{theorem}\label{thannulus}
	Let $n\ge2$, $0<r<R$ and $\O=B_R\setminus\overline{B_r}$. Then, for every $\tau\ge0$, the principal eigenfunction~$\vp_\tau$ solving~\eqref{evm} is radially symmetric, i.e., $\pt(x)=\Phi_\tau(|x|)$ in $\ol\O$ for some function $\Phi_\tau:[r,R]\to[0,1]$, and
	$$-\frac{v_{\tau}(x)}{\tau}=\frac{\nabla\pt(x)}{|\nabla\pt(x)|}\to\nabla d(x)\ \hbox{ as }\tau\to+\infty\ \hbox{ for all }x\in\overline{\O}\hbox{ with }|x|\neq\frac{r+R}{2}.$$
\end{theorem}

Before going into the proofs, let us give some heuristic analytic and probabilistic interpretations of these main results. On the one hand, minimizing the principal eigenvalue $\lambda_v$ of~\eqref{ev} means minimizing the deleterious effect of the Dirichlet boundary condition, which makes the solutions of the related Cauchy problem
\begin{equation}\label{Cauchy}
u_t+v\cdot\nabla u=\Delta u,\qquad\!\left.\!\hspace{1pt}u\right|_{\partial\Omega}=0,
\end{equation}
converge to $0$ as $t\to+\infty$ with at least a rate of the type $e^{-\lambda_vt}$. The solutions diffuse and are transported along the vector field $v$. When $v$ points from the boundary to the center of the domain, the boundary value $0$ tends to be propagated inside the domain, making the solutions converge to $0$ as $t\to+\infty$ with a faster decay rate $\lambda_v$. Therefore, in order to minimize these negative effects, the vector fields $v_{\tau}=-\tau\nabla\pt/|\nabla\pt|$ minimizing~\eqref{lt} should better point towards the boundary. Our results make this formal statement rigorous and quantitative: the minimizing vector fields have the largest possible magnitude and their direction is parallel and opposite to that of the gradient of the distance to the boundary in neighborhoods of the boundary. The size of the boundary layer is precisely estimated and the asymptotic profile is found. Furthermore, outside of this boundary layer, the eigenfunctions become approximately constant and their maximal points are as far as possible from the boundary.

On the other hand, \eqref{Cauchy} coincides with the Kolmogorov equation of the stochastic differential equation
$$dX(t)=-v\,dt+\sqrt 2\, dW(t),$$
subject to {\em absorbtion} on $\partial\Omega$. Namely, the solution of \eqref{Cauchy} with a nonnegative initial datum $u_0$ is given by
$$u(x,t)=\mathbb{E}[u_0(X(t))\, |\, X(0)=x],$$
in which the contribution of the trajectories $X$ hitting the boundary of $\Omega$ before the time $t$ is equal to $0$. Thus in order to minimize the decay rate $\lambda_v$ of $u$ one should prevent the trajectories from hitting $\partial\Omega$. One could think that the best strategy to do this would be taking $-v$ pointing in the direction opposed to the closest point of the boundary at the maximum allowed intensity $\tau$, that is, $-v_\tau=\tau\nabla d$. This is not the case because the presence of the Brownian motion ``deviates'' the trajectory from the imposed drift $-v$ and makes the optimization problem a nonlocal one which depends on the whole boundary of $\Omega$. However, as $\tau$ goes to $+\infty$, the influence of the Brownian motion becomes negligible compared with the drift, leading to the conjecture that~$-v_\tau$ aligns with $\nabla d$ in the limit.

Let us now quickly describe the proofs. For Theorem~\ref{thmaxima}, we compare $\vp_{\tau}$ with the function $u(x):=1-e^{-\gamma\tau(d(x)+\varepsilon)}$, where $\gamma>1$ and $\varepsilon>0$. Namely, for all $\tau$ large enough, there holds $-\Delta u-\tau\left\vert \nabla u\right\vert\ge\gamma\tau e^{-\gamma\tau(d+\varepsilon)}$ in the viscosity sense in $\O$. Then, looking at the points where $\frac{\vp_{\tau}}{u}$ reaches its maximum in~$\O$, it follows from the asymptotic behavior of $\lambda(\tau)$ recalled above that, $\vp_\tau\leq \frac{1-e^{-\gamma\tau d}}{1-e^{-\tau R/\gamma}}$ in $\Omega$ for $\tau$ large enough, with $R:=\max_{\overline{\O}}d$. This yields the desired conclusion, since the maximum of $\vp_{\tau}$ in $\overline{\O}$ is $1$. The major ingredient in this argument is the semiconcavity of $d$ (see Section~\ref{loc} below).

As far as Theorem~\ref{thconv} is concerned, one readily deduces from the inequality 
$\vp_\tau(x)\leq \frac{1-e^{-\gamma\tau d(x)}}{1-e^{-{\tau R/\gamma}}}$ seen before that
$$\limsup_{\tau\to+\infty}\left(\sup_{x\in\O}\frac{\pt(x)}{1-e^{-\tau d(x)}}\right)\leq1.$$
In order to derive a corresponding lower bound, 
a preliminary step is to prove that $\vp_{\tau}$ converges to~$1$ as $\tau\rightarrow +\infty$ locally uniformly in $\Omega$. This is first shown in a neighborhood of $x_\tau$ where $\vp_{\tau}(x_\tau)=1$, by means of a blow-up of the function $\vp_{\tau}$ around $x_{\tau}$ that can be performed thanks to Theorem \ref{thmaxima}.
Next we propagate this convergence result up to a boundary layer of width of order $1/\tau$
using a technical lemma about the variation of the minimum on concentric spheres,
together with a covering argument.
The convergence result in the whole $\Omega$ is eventually obtained by comparison with 
a lower barrier function of the type $x\mapsto 1-e^{-\eta\tau(d(x)-\e)}$.

Finally, for the proof of Theorem~\ref{thnabla}, we argue by contradiction, using a rescaling procedure of $\vp_{\tau}$ again.

In the present paper we consider the problem of minimizing the principal eigenvalue 
$\lambda_v$ in \eqref{ev}, but one could also be interested in maximizing such eigenvalue. 
That is, one could define $\ol{\lambda}(\tau)$ by replacing the ``$\inf$'' with a ``$\sup$'' in 
\eqref{lt}. It turns out that also $\ol{\lambda}(\tau)$ is attained by 
a unique vector field $\ol{v}_\tau$, with the principal eigenfunctions $\ol{\varphi}_\tau$
satisfying \eqref{vmin} and \eqref{evm} with $-\tau$ replaced by $\tau$.
In such case we conjecture that, after the $L^\infty$ normalization, 
$\ol{\varphi}_\tau$ converges as $\tau\to+\infty$ locally uniformly to $0$ in the complement of
the set where $d$ reaches its maximum.

The paper is organized as follows. Some basic properties of $\lambda(\tau)$ are given in Section~\ref{prel}. We show Theorem~\ref{thmaxima} in Section~\ref{loc}. Section~\ref{convfar} is devoted to the proof of the local uniform convergence of $\vp_{\tau}$ to $1$, and the proof of Theorem~\ref{thconv} is completed in Section~\ref{asymprof}. Section~\ref{asympgrad} is concerned with the proof of Theorem~\ref{thnabla}. We complete the paper with Section~\ref{secradial} and the proof of Theorem~\ref{thannulus} in the case when $\O$ is an annulus.


\SE{Preliminary properties of $\lt$}\label{prel}

We show in this section some basic properties of the minimizing principal eigenvalue~$\lt$, which will be used in the next sections.

\begin{proposition}\label{pro:char}
There holds
\Fi{char}
\begin{split}
\lt &=\max\big\{\lambda\ : \exists\, \vp\in C^2(\O),\ \vp>0 \text{ and } -\Delta\vp-\tau|\nabla\vp|\geq\lambda\vp \text{ in }\O\big\}\vspace{3pt}\\
&=\min\big\{\lambda\ : \exists\, \vp\in C^2(\O)\cap C^1(\ol\O),\ \vp>0\text{ in }\O,\vspace{3pt}\\ 
& \qquad\qquad\qquad\qquad\ -\Delta\vp-\tau|\nabla\vp|\leq\lambda\vp \text{ in }\O\text{ and }\vp=0 \text{ on }\partial\O\big\}.	
\end{split}
\Ff
Moreover, the $\max$ and the $\min$ are achieved only by $\vp=\pt$ $($up to scalar multiples$)$, and $\lt$ is strictly decreasing with respect to $\tau$ as well as to the inclusion of domains~$\O$.
\end{proposition}

\begin{proof}
Let $\lambda\in\R$ and $\vp\in C^2(\O)$ be such that $\vp>0$ and
$$-\Delta\vp-\tau|\nabla\vp|\geq\lambda\vp\ \hbox{ in }\O.$$
Then, for any $v$ such that $\| v\|_{L^\infty(\O)}\leq\tau$ there holds
$$-\Delta\vp+v\.\nabla\vp\geq\lambda\vp\quad \text{ in }\O.$$
By~\cite{BNV}, one knows that
$$\lambda_v=\max\big\{\mu\ : \exists\, \phi\in C^2(\O),\ \phi>0 \text{ and } -\Delta\phi+v\cdot\nabla\phi\geq\mu\phi \text{ in }\O\big\},$$
hence $\lambda\leq \lambda_v$, and that the
equality $\lambda=\lambda_v$ holds if and only if $\vp$ is the principal eigenfunction of~\eqref{ev}. Taking $v=v_\tau$ given by~\eqref{vmin} we get $\lambda\le\lambda_{v_\tau}=\lt$
and then using the fact that the function $\pt$ satisfies~\eqref{evm} we derive the first equality in~\eqref{char}. 
Moreover, $\lambda<\lt$ unless $\vp$ coincides with $\pt$ up to a scalar multiple.

The first equality in~\eqref{char} also implies that $\lt$ is non-increasing with respect to both $\tau$ and the inclusion of domains $\O$ (the former also immediately follows from the definition~\eqref{lt}). Consider now $\tau'\in\R$ such that $\tau'>\tau$ and assume by contradiction that $\lambda(\tau')=\lambda(\tau)$. Since $|\nabla\vp_{\tau'}|>0$ a.e. in $\O$, one then gets
\be\label{strict}
-\Delta\vp_{\tau'}-\tau|\nabla\vp_{\tau'}|>-\Delta\vp_{\tau'}-\tau'|\nabla\vp_{\tau'}|=\lambda(\tau')\,\vp_{\tau'}=\lambda(\tau)\,\vp_{\tau'}\hbox{ a.e. in }\Omega.
\ee
Hence, $\vp:=\vp_{\tau'}$ is a supersolution of~\eqref{evm} with $\lambda:=\lambda(\tau)$. The conclusion of the previous paragraph implies that $\vp_{\tau'}$ and $\pt$ must coincide up to a scalar multiple, which contradicts the strict inequality in~\eqref{strict}. Therefore, $\lambda(\tau')<\lt$. Similarly, the solution $\vp'$ of the nonlinear eigenvalue problem~\eqref{evm} with $\O$ replaced by a (connected) domain $\O'\supsetneq\O$ and with principal eigenvalue $\lambda'$ satisfies $-\Delta\vp'-\tau|\nabla\vp'|=\lambda'\vp'$ in~$\Omega$. Since $\max_{\partial\O}\vp'>0$, $\vp'$ does not coincide with $\vp_\tau$ in $\O$ up to a scalar multiple and we conclude from the previous paragraph that $\lambda'<\lt$. This means that the monotonicities of $\lt$ with respect to both $\tau$ and $\O$ are strict.

For the second equality in~\eqref{char}, consider $\lambda\in\R$ and $\vp\in C^2(\Omega)\cap C^1(\overline{\Omega})$ be such that $\vp>0$ in $\O$, $-\Delta\vp-\tau|\nabla\vp|\leq\lambda\vp$ in $\O$ and $\vp=0$ on $\partial\O$. Define 
$$v(x):=\begin{cases}\displaystyle-\tau\frac{\nabla\vp(x)}{|\nabla\vp(x)|}\ & 
\text{if }|\nabla\vp(x)|>0,\vspace{3pt}\\
0 & \text{otherwise}.
\end{cases}$$
The function $\vp$ is a positive subsolution of the eigenvalue problem 
$$-\Delta\vp+v\.\nabla\vp=\lambda\vp \quad \text{in }\O,\qquad\vp=0  \quad \text{on }\partial\O.$$
By~\cite{BNV}, this means that $\lambda_{v}\leq\lambda$ and that equality holds if and only if $\vp$ is the principal eigenfunction for the above problem. Since $\|v\|_{L^{\infty}(\O)}\le\tau$, it follows that $\lt\leq\lambda$ and, by uniqueness of the vector field $\ul v$ minimizing~\eqref{lt}, one infers that equality $\lt=\lambda$ holds if and only if $\vp$ coincides with the principal eigenfunction $\pt$ up to a scalar multiple. The proof of Proposition~\ref{pro:char} is thereby complete.
\end{proof}


\SE{Location of maxima}\label{loc}

This section is devoted to the proof of Theorem~\ref{thmaxima}. We first recall from Section~\ref{intro} that $d:\ol\O\to[0,+\infty)$ denotes the distance function from $\partial\O$. Then define the open neighborhood of $\partial\Omega$ of width $\delta>0$, relatively to $\O$, as in~\eqref{Odelta}, that is,
$$\Omega^{\delta}=\big\{x\in\O\;:\; d(x)<\delta\big\}.$$

We will make use of some properties of the distance function in order to construct a family of supersolutions to \eqref{evm}. One of them is the {\em semiconcavity}, which is a straightforward consequence of the regularity of $d$ in a neighborhood of $\partial\O$, as shown for instance in \cite{CanSin}. We include  the proof of this fact below because we use a different notion of semiconcavity, expressed in terms of the {\em second order sub-differential jet}. For a continuous function $u$ at a point $x$, the latter is defined by
\[\begin{split}
J^-u(x):=\Big\{(p,X)\in\R^N\times \mathbb{S}^N\;:\; u(y)\geq & \ u(x)+p\.(y-x)+\frac12 X(y-x)\cdot(y-x)\\
&+o(|y-x|^2)\quad \text{as }y\to x\Big\},
\end{split}\]
where $\mathbb{S}^N$ denotes the space of $N\times N$ symmetric matrices and $Xz\cdot z$ denotes $Xz\cdot z=\sum_{1\le i,j\le N}X_{ij}z_iz_j$ for $X=(X_{ij})_{1\le i,j\le N}\in\mathbb{S}^N$ and $z=(z_i)_{1\le i\le N}\in\R^N$.

\begin{lemma}\label{lem:dist}
There exists $K>0$ such that
$$\forall\, x\in\Omega,\ \forall\,(p,X)\in J^-d(x),\ \big[|p|=1\mbox{ and } X\leq K\,I_N\big],$$
where $I_N$ is the identity matrix and $\leq$ is the usual order on $\mathbb{S}^N$.
\end{lemma}

\begin{proof}
First of all, we know from \cite[Lemmas~14.16 and 4.17]{GT} that there exists $\delta>0$ such that $d$ is of class $C^2(\overline{\Omega^\delta})$ and therefore has bounded Hessian matrix in $\O^\delta$. Consider now $x\in\O\setminus\O^\delta$ and let $z\in\partial\O$ be its or one of its projections on $\partial\O$, that is, $|x-z|=d(x)$. The function $g(y):=|y-z|$ satisfies $g(x)=d(x)$ and, for all $y\in\O$, $g(y)\geq d(y)$. As a consequence, if there exists $(p,X)\in J^-d(x)$ then for $y\in\O$ there holds
\[\begin{split}
g(y)\geq d(y) 
&\geq d(x)+p\.(y-x)+\frac12 X(y-x)\cdot(y-x)+o(|y-x|^2)\\
&=g(x)+p\.(y-x)+\frac12 X(y-x)\cdot(y-x)+o(|y-x|^2),
\end{split}\]
that is, $(p,X)\in J^-g(x)$. Since $g$ is a $C^2$ function outside $z$, we infer that
$$p=\nabla g(x)=\frac{x-z}{|x-z|},\qquad X=D^2g(x)=\frac1{|x-z|}I_N-\frac1{|x-z|^3}(x-z)\otimes(x-z).$$
It follows that $X\leq(d(x))^{-1}I_N\leq\delta^{-1}I_N$. This concludes the proof of the lemma.
\end{proof}

\begin{proof}[Proof of Theorem~\ref{thmaxima}]
Consider an arbitrary $\gamma>1$. For $\tau>0$ and $\e>0$, define the function 
\begin{equation}\label{upperbarrier}
u_{\tau,\e}(x)=1-e^{-\gamma\tau(d(x)+\e)}.
\end{equation}
If $d$ were a $C^2$ function in the whole $\Omega$ we would have
$$-\Delta u_{\tau,\e}-\tau|\nabla u_{\tau,\e}|=\gamma\tau e^{-\gamma\tau(d+\e)}\big(-\Delta d+\gamma\tau|\nabla d|^2-\tau|\nabla d|\big).$$
It is easy to check, using the fact that $u_{\tau,\e}$ is the composition of the function $d$ with a  strictly increasing smooth function, that the above formal computation holds for the sub-differential jets, in the sense of multivalued functions, i.e.,
\[\begin{array}{l}
\big\{-\Tr Y-\tau|q|\ :\ (q,Y)\in J^- u_{\tau,\e}(x)\big\}\vspace{3pt}\\
\qquad\qquad=\big\{  \gamma\tau e^{-\gamma\tau(d(x)+\e)}\big(-\Tr X+\gamma\tau|p|^2-\tau|p|\big): (p,X)\in J^- d(x)\big\}
\end{array}\]
for every $x\in\O$. Owing to Lemma~\ref{lem:dist} and because $\gamma>1$, there is $\tau_0>0$ such that, for every $\tau\ge\tau_0$, $x\in\Omega$ and $(p,X)\in J^- d(x)$, one has $-\Tr X+\gamma\tau|p|^2-\tau|p|\ge1$. It follows that $-\Delta u_{\tau,\e}-\tau|\nabla u_{\tau,\e}|\geq\gamma\tau e^{-\gamma\tau(d+\e)}$ in $\Omega$, still in the viscosity sense, that is,
\Fi{u>}
\forall\tau\geq\tau_0,\,\forall\epsilon>0,\;\forall x\in\Omega,\,\forall(q,Y)\in J^-u_{\tau,\e}(x),\ -\Tr Y\!-\!\tau|q|\!\ge\!\gamma\tau e^{-\gamma\tau(d(x)+\e)}.
\Ff

Next, define
$$k_{\tau,\e}:=\sup_{\O}\frac{\vp_\tau}{u_{\tau,\e}}.$$
Because $u_{\tau,\e}>0$ in $\ol\O$ together with $\vp_\tau>0$ in $\Omega$ and $\vp_\tau=0$ on $\partial\O$, the above supremum is actually a maximum, attained at some $x_{\tau,\e}\in\O$. Namely, the function~$k_{\tau,\e}^{-1}\vp_\tau$ touches $u_{\tau,\e}$ from below at $x_{\tau,\e}$, whence 
$$k_{\tau,\e}^{-1}(\nabla\varphi_\tau,D^2\varphi_\tau)\in J^-u_{\tau,\e}(x_{\tau,\e}).$$
We then deduce from \eqref{u>} that
$$\forall\, \tau\geq\tau_0,\ \forall\,\e>0,\quad -\Delta \vp_\tau(x_{\tau,\e})-\tau|\nabla \vp_\tau(x_{\tau,\e})|\geq k_{\tau,\e}\gamma\tau e^{-\gamma\tau(d(x_{\tau,\e})+\e)},$$
this time in the classical sense. On the other hand, by \eqref{evm},
$$-\Delta \vp_\tau(x_{\tau,\e})-\tau|\nabla \vp_\tau(x_{\tau,\e})|=\lambda(\tau)\vp_\tau(x_{\tau,\e}),$$
and therefore
\Fi{lt>}
\forall\, \tau\geq\tau_0,\ \forall\,\e>0,\quad 
\gamma\tau e^{-\gamma\tau(d(x_{\tau,\e})+\e)}\leq 
\lambda(\tau)\frac{\vp_\tau(x_{\tau,\e})}{k_{\tau,\e}}
=\lambda(\tau)u_{\tau,\e}(x_{\tau,\e})<\lambda(\tau).
\Ff

Now, let us call $R:=\max_{\ol\O}d>0$. By the monotonicity with respect to the inclusion of domains, the quantity $\lambda(\tau)$ is bounded from above by the one corresponding to the case $\Omega=B_R$, and we know from~\cite[Lemma 7.2]{HNR} that the logarithm of the latter behaves like $-\tau R$ as $\tau\rightarrow +\infty$. We can then find $\tau_1>0$ such that $\lambda(\tau)\leq e^{-\tau R/\gamma}$ for~$\tau\geq\tau_1$. We can assume without loss of generality that $\tau_1\geq\max(\tau_0,1)$. It then follows from \eqref{lt>} that
$$\forall\, \tau\geq\tau_1,\ \forall\,\e>0,\quad d(x_{\tau,\e})+\e\geq\frac R{\gamma^2}.$$
Recalling the definition of $k_{\tau,\e}$, we eventually derive, for all $\tau\geq\tau_1$ and $\e>0$,
$$\forall\,x\in\O,\quad \frac{\vp_\tau(x)}{u_{\tau,\e}(x)}\leq k_{\tau,\e}=\frac{\vp_\tau(x_{\tau,\e})}{u_{\tau,\e}(x_{\tau,\e})}\leq \frac1{1-e^{-\gamma\tau(d(x_{\tau,\e})+\e)}}\leq \frac1{1-e^{-\tau R/\gamma}}.$$
Whence, passing to the limit as $\e\to0$ yields
\Fi{pt<}
\forall\,\tau\geq\tau_1,\ \forall\,x\in\O,\quad \vp_\tau(x)\leq \frac{1-e^{-\gamma\tau d(x)}}{1-e^{-\tau R/\gamma}}.
\Ff
It follows that, for $\tau\geq\tau_1$, the function $\pt$ cannot attain its maximal value 1 at any point $x$ with $d(x)<R/\gamma^2$. This concludes the proof by the arbitrariness of $\gamma>1$.
\end{proof}


\SE{Local uniform convergence}\label{convfar}

In this section we derive the locally uniform convergence of the functions $\pt$ to $1$ as $\tau\to+\infty$ (see Lemma~\ref{lem:1} below), which is part of Theorem~\ref{thconv}. To do so, we first show that the minimum of $\pt$ in some concentric balls strongly included in $\O$ can be controlled by some constants close to $1$ as $\tau$ is large.

\begin{lemma}\label{lem:minima}
Let $x_0\in\O$ and $0<R<R'$ be such that $B_{R'}(x_0)\subset\O$ and let $\e\in(0,1/R)$. Then there exists $\tau_0>0$, only depending on $n$, $R$, $R'$ and~$\e$, such that
\Fi{minima}
\forall\,\tau\ge\tau_0,\ \forall\ \frac{2(n-1)}\tau\leq r\leq r'\leq R,\qquad\min_{\partial B_{r'}(x_0)}\vp_\tau\geq \frac{1-\e r'}{1-\e r}\,\min_{\partial B_r(x_0)}\vp_\tau.
\Ff
\end{lemma}

\begin{remark} For any $\tau\ge0$, since the function $\pt\in C^2(\O)\cap C(\overline{\O})$ satisfies
$$-\Delta\pt=\tau|\nabla\pt|+\lt\pt>0\ \hbox{ in }\O,$$
the maximum principle implies that
\begin{equation}\label{maxprin}
\min_{\overline{\omega}}\pt=\min_{\partial\omega}\pt
\end{equation}
for any non-empty subset $\omega\subset\Omega$. In particular,~\eqref{minima} can be rewritten as
$$\min_{\overline{B_{r'}(x_0)}}\vp_\tau\geq \frac{1-\e r'}{1-\e r}\,\min_{\overline{B_r(x_0)}}\vp_\tau$$
for all $\tau\ge\tau_0$ and $2(n-1)/\tau\leq r\leq r'\leq R$.
\end{remark}

\begin{proof}[Proof of Lemma~\ref{lem:minima}]
Let us assume for simplicity that $x_0=0$ and let $R$, $R'$ and $\e$ be as in the statement. Consider a smooth function $\chi$ which is radially symmetric, nonincreasing in the radial direction and satisfies
$$\chi(x)=1\ \text{ for }|x|\leq R, \qquad\chi(x)=0\ \text{ for }|x|\geq R'.$$
Then define
$$u(x):=\chi(x)-\e|x|$$
for all $x\in\R^n$. The set where this function is positive is equal to some ball $B_{R_\e}$, with
$$0<R<R_\e<R'.$$
Now, for any $\tau>0$ such that $2(n-1)/\tau<R_\e$, there holds, in the annulus 
$\overline{B_{R_\e}}\setminus B_{2(n-1)/\tau}$,
$$-\Delta u-\tau|\nabla u|\leq -\Delta\chi-\e\left(\tau-\frac{n-1}{|x|}\right)\leq-\Delta\chi-\frac{\e\tau}2.$$
Hence, there exists $\tau_0>0$ large enough, depending on $n$, $R_\e$, $\chi$ and $\epsilon$ (hence, depending on $n$, $R$, $R'$ and $\e$) such that
\begin{equation}\label{tau0}
\frac{2(n-1)}{\tau}<R_\e\ \hbox{ and }\ -\Delta u-\tau|\nabla u|<0\hbox{ in }\overline{B_{R_\e}}\setminus B_{2(n-1)/\tau},\ \hbox{ for all }\tau\ge\tau_0.
\end{equation}

Fix any $\tau\ge\tau_0$, any $r\in [2(n-1)/\tau,R]$ (hence, $\e r\le\e R<1$), and set
$$k:=\frac{1}{1-\e r}\,\min_{\partial B_r}\vp_\tau>0$$
Suppose by contradiction that $\pt-k u<0$ somewhere in the annulus $\ol{B_{R_\e}}\setminus B_{r}$. Since $\pt-k u$ is nonnegative on the boundary of this set (we recall that $u=0$ on $\partial B_{R_\e}$ and $u=1-\e r$ on $\partial B_r$), the negative minimum is reached at some interior point $\hat x\in B_{R_\e}\setminus\overline{B_r}$, and thus there holds
$$\lt\pt(\hat x)=-\Delta\pt(\hat x)-\tau|\nabla \pt(\hat x)|\leq -k\Delta u(\hat x)-k\tau|\nabla u(\hat x)|<0$$
by~\eqref{tau0}. This is impossible because $\lt>0$. As a consequence,
$$\pt-k u\geq 0\ \hbox{ in }\ol{B_{R_\e}}\setminus B_{r},$$
from which~\eqref{minima} is readily obtained recalling that $\chi=1$ in $B_R$.
\end{proof}

By making use of Theorem~\ref{thmaxima} and Lemma~\ref{lem:minima} together with a covering argument, the following local uniform limit holds as $\tau\to+\infty$:

\begin{lemma}\label{lem:1}
The family $(\vp_\tau)_{\tau>0}$ converges locally uniformly to $1$ in $\O$  as $\tau\to+\infty$.
\end{lemma}

\begin{proof}
Let $(x_\tau)_{\tau>0}$ be a family of maximal points of $(\vp_\tau)_{\tau>0}$, i.e.~such that $\pt(x_\tau)=1$. Remember from Theorem~\ref{thmaxima} that $d(x_\tau)\to\max_{\ol\O}d>0$ as $\tau+\infty$. Firstly, we show that  $\vp_\tau\to1$ near $x_\tau$ as $\tau\to+\infty$. Consider the family of functions~$(\psi_\tau)_{\tau>0}$ defined~by
$$\psi_\tau(x):=\pt\left(x_\tau+\frac{x}\tau\right).$$
Since the points $x_\tau$ are bounded away from $\partial\O$ as $\tau\to+\infty$, the functions $(\psi_\tau)_{\tau>0}$ are defined in a family of domains which converge to the whole space $\R^n$ as $\tau\to+\infty$, and satisfy there
$$-\Delta\psi_\tau-|\nabla\psi_\tau|=\frac{\lt}{\tau^2}\,\psi_\tau.$$
Therefore, since the eigenvalues $\lt$ are bounded as $\tau\to+\infty$ by Proposition~\ref{pro:char} (and even converge to $0$ as $\tau\to+\infty$ as recalled in Section~\ref{intro}), interior elliptic estimates~\footnote{One first writes the equation in linear form, with first order coefficient equal to $-\nabla\psi_\tau/|\nabla\psi_\tau|$, which has $L^\infty$ norm equal to 1, and derives $W^{2,p}_{loc}$ a priori estimates for any $p\in(1,+\infty)$; then, by Morrey's inequality, the terms $|\nabla\psi_\tau|$ can be considered as bounded data in $C^{0,\alpha}_{loc}$, for any $\alpha\in(0,1)$, leading to $C^{2,\alpha}_{loc}$ estimates.} imply that, as $\tau\to+\infty$, the functions $\psi_\tau$ converge in $C^2_{loc}(\R^n)$, at least for a sequence $(\tau_n)_{n\in\N}\to+\infty$,  to some function $\psi_\infty$ satisfying
$$\psi_\infty(0)=1=\max_{\R^n}\psi_\infty,\qquad -\Delta\psi_\infty-|\nabla\psi_\infty|=0\hbox{ in }\R^n.$$
Since the latter equation can be written in linear form, with a bounded first order coefficient and without zero order term, it follows from the strong maximum principle that $\psi_\infty\equiv1$ in $\R^n$. By uniqueness of the limit, this means that the whole family $(\psi_\tau)_{\tau>0}$ converges locally uniformly to $1$ as $\tau\to+\infty$. Namely, 
\Fi{loc1}
\forall\, M\ge0,\qquad \|\pt-1\|_{L^\infty(B_{M/\tau}(x_\tau))}\to0\quad\text{as }\tau\to+\infty.
\Ff

Next, we extend the above convergence to compact subsets of $\O$ using Lemma~\ref{lem:minima} together with a covering argument. Consider now
$$\O\setminus\O^\delta=\big\{x\in\O\ :\ d(x)\ge\delta\big\}.$$
Because of the regularity of $\O$, this compact set is non-empty, smooth and connected for $\delta>0$ sufficiently small. Fix in the sequel any value of $\delta\in(0,\max_{\ol\O}d)$ for which this holds and take any $\e>0$ such that
$$0<\e<\frac{2}{\delta}.$$

We now apply Lemma~\ref{lem:minima} with this value of $\e$, any point $x_0=y\in\O\setminus\O^\delta$, $R=\delta/2$ and $R'=\delta$. It implies that, for $\tau$ larger
than some $\tau_1$, only depending on $n$, $\delta$ and $\e$, there holds:
$$\forall\,y\in \O\setminus\O^\delta,\ \forall\,\frac{2(n-1)}{\tau}\leq r\leq r'\leq\frac\delta2,\qquad\min_{\partial B_{r'}(y)}\vp_\tau\geq\Big(1-\frac{\e\delta}{2}\Big)\,\min_{\partial B_r(y)}\vp_\tau,$$
from which, taking $r=2(n-1)/\tau$, $r'=\delta/2$ and using~\eqref{maxprin}, we deduce
\Fi{annulus}
\forall\,y\in \O\setminus\O^\delta,\qquad\min_{\ol{B_{\delta/2}(y)}}\vp_\tau\geq\Big(1-\frac{\e\delta}{2}\Big)\,\min_{\ol{B_{2(n-1)/\tau}(y)}}\vp_\tau.
\Ff
On the other hand, using~\eqref{loc1} with $M=2(n-1)$ we get, for $\tau>0$ sufficiently large, 
$$\min_{\ol{B_{2(n-1)/\tau}(x_\tau)}}\vp_\tau\geq 1-\frac{\e\delta}{2},$$
whence, for possibly larger $\tau\geq\tau_1$ so that $x_\tau\in\O\setminus\O^\delta$ (remember that $\lim_{\tau\to+\infty}d(x_\tau)=\max_{\ol\O}d>\delta$),
we can gather the above inequality together with~\eqref{annulus} and derive
\Fi{step0}
\min_{\ol{B_{\delta/2}(x_\tau)}}\vp_\tau\geq\Big(1-\frac{\e\delta}{2}\Big)^2.
\Ff

Consider then a covering of $\O\setminus {\O^\delta}$ of the type
$$\O\setminus {\O^\delta}\ \subset\ B_{\delta/6}(y_1)\cup\cdots\cup B_{\delta/6}(y_m)$$
with $y_1,\dots,y_m\in\O\setminus {\O^\delta}$. Let $\tau\geq\tau_1$ 
be large enough so that $x_\tau\in\O\setminus{\O^\delta}$, property~\eqref{step0} holds and, in addition, $2(n-1)/\tau\le\delta/6$. Since $\tau\geq\tau_1$, we can apply~\eqref{annulus} and deduce that
\Fi{half}
\forall\,j=1,\dots,m,\qquad\min_{\ol{B_{\delta/2}(y_j)}}\vp_\tau\geq\Big(1-\frac{\e\delta}{2}\Big)\,\min_{\ol{B_{\delta/6}(y_j)}}\vp_\tau.
\Ff
Moreover, because $x_\tau\in\O\setminus{\O^\delta}$, there exists $j_1\in\{1,\dots,m\}$ such that $x_\tau\in B_{\delta/6}(y_{j_1})$ and, in particular, $B_{\delta/6}(y_{j_1})\subset B_{\delta/2}(x_\tau)$. Combining~\eqref{step0} and~\eqref{half} we then get
$$\min_{\ol{B_{\delta/2}(y_{j_1})}}\vp_\tau\geq\Big(1-\frac{\e\delta}{2}\Big)^3.$$
Now, because $\O\setminus{\O^\delta}$ is connected, there exists $j_2\neq j_1$ such that
$$B_{\delta/6}(y_{j_2})\cap B_{\delta/6}(y_{j_1})\neq\emptyset$$
which entails $B_{\delta/6}(y_{j_2})\subset B_{\delta/2}(y_{j_1})$. Thus, as before, we find
$$\min_{\ol{B_{\delta/2}(y_{j_2})}}\vp_\tau\geq\Big(1-\frac{\e\delta}{2}\Big)^4.$$
By a recursive argument we then find a permutation $\{j_1,\dots,j_m\}$ of $\{1,\dots,m\}$ satis\-fying the following property:
$$\forall\,k=2,\dots,m,\quad B_{\delta/6}(y_{j_k})\cap\left( \bigcup_{l=1}^{k-1} B_{\delta/6}(y_{j_l})\right)\neq\emptyset\ \ \hbox{ and }\ \ \min_{\ol{B_{\delta/2}(y_{j_k})}}\vp_\tau\geq\Big(1-\frac{\e\delta}{2}\Big)^{2+k}.$$
It follows that
$$\min_{\O\setminus {\O^\delta}}\vp_\tau\geq\Big(1-\frac{\e\delta}{2}\Big)^{2+m}.$$

This concludes the proof because $\delta>0$ can be chosen arbitrarily small, providing an integer $m$, and next one can take $\e>0$ as small as wanted.
\end{proof}


\SE{Asymptotic profile near the boundary}\label{asymprof}

This section is devoted to the proof of Theorem~\ref{thconv}, namely the limit~\eqref{profile}. Since Lemma~\ref{lem:1} provides the limit in any compact subset of $\O$, we only have to show the asymptotic profile~\eqref{profile} in a neighborhood of $\partial\O$. To do so, we will make use of the upper bound~\eqref{pt<} derived before and of a lower barrier function of the same type as the upper barrier function~\eqref{upperbarrier} used in the proof of Theorem~\ref{thmaxima}.

\begin{proof}[Proof of Theorem~\ref{thconv}]
The upper bound follows from~\eqref{pt<}. Indeed, for any given $\gamma>1$, 
we have that~\eqref{pt<} holds for $\tau_1>0$ sufficiently large (depending on $\O$ and $\gamma$). It follows that
$$\limsup_{\tau\to+\infty}\left(\sup_{x\in\O}\frac{\pt(x)}{1-e^{-\gamma\tau d(x)}}\right)\leq1.$$
Now, since the function $\chi$ defined by 
\begin{equation}\label{defchi}
\chi(s)=\frac{1-s^\gamma}{1-s}
\end{equation}
is increasing on $(0,1)$ and tends to $\gamma$ as $s\to1^-$, we deduce that 
$$\limsup_{\tau\to+\infty}\left(\sup_{x\in\O}\frac{\pt(x)}{1\!-\!e^{-\tau d(x)}}\right)\le\limsup_{\tau\to+\infty}\left[\left(\sup_{x\in\O}\frac{\pt(x)}{1\!-\!e^{-\gamma\tau d(x)}}\right)\!\times\!\left(\sup_{x\in\O}\frac{1\!-\!e^{-\gamma\tau d(x)}}{1\!-\!e^{-\tau d(x)}}\right)\right]\leq\gamma,$$
which provides us with the desired upper bound, due to the arbitrariness of $\gamma>1$.

It remains to derive the lower bound. To do this, from \cite[Lemmas~14.16 and 14.17]{GT}, we fix $\delta>0$ such that $d$ is of class $C^2$ in $\overline{\O^\delta}$, with $\O^\delta$ defined by~\eqref{Odelta}, and we consider any $0<\eta<1$. Then, we define the following family of functions, which will play the role of lower barriers, for $\tau>0$ and $\e>0$:
$$w_{\tau,\e}(x)=1-e^{-\eta\tau(d(x)-\e)}.$$
These functions satisfy
$$-\Delta w_{\tau,\e}-\tau|\nabla w_{\tau,\e}|=\eta\tau e^{-\eta\tau(d(x)-\e)}\left(-\Delta d+\tau\left(\eta-1\right)\right)\quad\text{in }\ol{\O^\delta}.$$
Therefore, there exists $\tau_0>0$, only depending on $\O$, $\eta$ and $\delta$, such that
\Fi{w<}
\forall\,\tau\ge\tau_0,\ \forall\,\e>0,\ \ -\Delta w_{\tau,\e}-\tau|\nabla w_{\tau,\e}|<0\quad\text{in }\ol{\O^\delta}.
\Ff

Let us call
$$k_{\tau,\e}:=\sup_{\O^\delta}\frac{w_{\tau,\e}}{\vp_\tau}.$$
For any $\tau>0$ and $0<\e<\delta$, the function $w_{\tau,\e}$ is nonpositive in $\ol{\O^\e}$ and positive in~$\ol{\O^\delta}\setminus\ol{\O^\e}$, thus the above supremum $k_{\tau,\e}$ is positive and is actually a maximum, attained at some $x_{\tau,\e}$ such that $\e<d(x_{\tau,\e})\le\delta$. We claim that
$d(x_{\tau,\e})=\delta$ for any $\tau\ge\tau_0$ and $\e\in(0,\delta)$. Indeed, otherwise $\e<d(x_{\tau,\e})<\delta$ and the point $x_{\tau,\e}$ would then be a local maximum of the function $w_{\tau,\e}-k_{\tau,\e}\pt$ which is interior to $\O^\delta\setminus\ol{\O^\e}$. Hence,
$$-\Delta w_{\tau,\e}(x_{\tau,\e})-\tau|\nabla w_{\tau,\e}(x_{\tau,\e})|\ge k_{\tau,\e}\big(-\Delta\pt(x_{\tau,\e})-\tau|\nabla\pt(x_{\tau,\e})|\big),$$
whence $0>k_{\tau,\e}\lambda(\tau)\pt(x_{\tau,\e})$ by~\eqref{w<} and~\eqref{evm}. This contradicts the positivity of $k_{\tau,\e}$, $\lambda(\tau)$ and $\pt(x_{\tau,\e})$.

Therefore, for all $\tau\geq\tau_0$ and $\e\in(0,\delta)$, we have that $d(x_{\tau,\e})=\delta$ and
$$\forall\,x\in\ol{\O^\delta},\quad\frac{w_{\tau,\e}(x)}{\vp_\tau(x)}\leq\frac{w_{\tau,\e}(x_{\tau,\e})}{\vp_\tau(x_{\tau,\e})}\leq\frac{1-e^{-\eta\tau\delta}}{\min_{d(y)=\delta}\vp_\tau(y)},$$
from which, letting $\e\to0$ we obtain
$$\forall\,x\in\ol{\O^\delta},\quad\frac{1-e^{-\eta\tau d(x)}}{\vp_\tau(x)}\leq\frac{1-e^{-\eta\tau\delta}}{\min_{d(y)=\delta}\vp_\tau(y)}.$$
By Lemma~\ref{lem:1}, the right-hand side converges to $1$ as $\tau\to+\infty$, and thus
$$\liminf_{\tau\to+\infty}\left(\inf_{x\in\O^\delta}\frac{\vp_\tau(x)}{1-e^{-\eta\tau d(x)}}\right)\geq1.$$
Finally, using again the monotonicity of the function $\chi$ defined in~\eqref{defchi} above, we easily deduce that
$$\liminf_{\tau\to+\infty}\left(\inf_{x\in\O^\delta}\frac{\vp_\tau(x)}{1-e^{-\tau d(x)}}\right)\geq\eta.$$

Because $\eta<1$ is arbitrary, we can replace it by $1$ in the above inequality. This provides the desired lower bound in $\O^\delta$. Observe that in the set $\O\setminus\O^\delta$, the uniform lower bound (as well as the upper bound) follows immediately from Lemma~\ref{lem:1}. The proof of Theorem~\ref{thconv} is thereby complete.
\end{proof}


\SE{Asymptotic behavior of $\nabla\pt$}\label{asympgrad}

In this section we prove Theorem~\ref{thnabla}. First of all, for any $y\in\partial\Omega$, we 
let $\nu(y)$ denote the outward unit normal to $\Omega$ at $y$. The proof of Theorem~\ref{thnabla} is based on the following lemma.

\begin{lemma}\label{convprofile}
For any sequence $(y_k)_{k\in\N}$ of $\partial\O$ converging to $y\in\partial\O$ and for any sequence $(\tau_k)_{k\in\N}$ of positive real numbers converging to $+\infty$, there holds
$$\varphi_{\tau_k}\Big(y_k+\frac{x}{\tau_k}\Big)\mathop{\longrightarrow}_{k\to+\infty}1-e^{x\cdot\nu(y)}\ \hbox{ for all }x\in\R^n\hbox{ such that }x\cdot\nu(y)<0.$$
\end{lemma}

\begin{proof}
The proof strongly relies on Theorem~\ref{thconv}. First of all, up to rotation and translation of the frame, one can assume without loss of generality that
$$y=0=(0,\cdots,0)\ \hbox{ and }\ \nu(y)=(0,\cdots,0,-1).$$
Fix then any $x=(x_1,\cdots,x_n)\in\R^n$ such that $x\cdot\nu(y)<0$, that is, $x_n>0$. Since $\nu(y_k)\to\nu(y)$ as $k\to+\infty$, it follows that $y_k+x/\tau_k$ belongs to $\Omega$ for all $k$ large enough (in particular, $d(y_k+x/\tau_k)>0$ for $k$ large enough), and, since $y_k\in\partial\O$,
\begin{equation}\label{tauk}
d\Big(y_k+\frac{x}{\tau_k}\Big)\le\frac{|x|}{\tau_k}\to0\ \hbox{ as }k\to+\infty.
\end{equation}
From formula~\eqref{profile} in Theorem~\ref{thconv}, it is therefore sufficient to show that
\begin{equation}\label{claim}
\tau_k\,d\Big(y_k+\frac{x}{\tau_k}\Big)\to x_n\ \hbox{ as }k\to+\infty
\end{equation}
in order to conclude the proof of Lemma~\ref{convprofile}.

Since $\nu(y)=(0,\cdots,0,-1)$ and $\partial\O$ is of class $C^2$, there exist $r>0$ and a $C^2$ function $g$ defined in a neighborhood $V$ of $(0,\cdots,0)\in\R^{n-1}$ such that 
$$\partial\O\cap B_r=\big\{(x',x_n)\ :\ x'\in V,\
x_n=g(x')\}$$
and
$$\frac{\partial g}{\partial x_i}(0)=0\ \hbox{ for all }1\le i\le n-1,$$
where, for all points $x=(x_1,\cdots,x_n)\in\R^n$, we write $x'=(x_1,\cdots,x_{n-1})$. Since $y_k\in\partial\O$ for every $k\in\N$ and $\lim_{k\to+\infty}d(y_k+x/\tau_k)=0$, one can assume without loss of generality that, for every $k\in\N$, there is a unique point $\xi_k\in\partial\O$ such that
$$d\Big(y_k+\frac{x}{\tau_k}\Big)=\Big|y_k+\frac{x}{\tau_k}-\xi_k\Big|.$$
Notice that, by~\eqref{tauk},
$$|\xi_k|\le\Big|\xi_k-y_k-\frac{x}{\tau_k}\Big|+\Big|y_k+\frac{x}{\tau_k}\Big|\le d\Big(y_k+\frac{x}{\tau_k}\Big)+|y_k|+\frac{|x|}{\tau_k}\le|y_k|+\frac{2|x|}{\tau_k}\to0$$
as $k\to+\infty$. For each $k\in\N$ large enough, the non-zero vector $y_k+x/\tau_k-\xi_k$ is parallel to the normal $\nu(\xi_k)$. Since for $k\in\N$ sufficiently large, the tangent space of~$\O$ at the point $\xi_k$ is generated by the vectors
$$T_i(\xi_k):=\Big(\underbrace{0,\cdots,0}_{i-1},1,\underbrace{0,\cdots,0}_{n-1-i},\frac{\partial g}{\partial x_i}(\xi'_k)\Big),\quad
i=1,\dots,n-1,$$
we find that $(y_k+x/\tau_k-\xi_k)\cdot T_i(\xi_k)=0$ for all $1\le i\le n-1$, that is,
$$y_{k,i}+\frac{x_i}{\tau_k}-\xi_{k,i}=-\frac{\partial g}{\partial x_i}(\xi'_k)\times\Big(y_{k,n}+\frac{x_n}{\tau_k}-\xi_{k,n}\Big).$$
Therefore,
\begin{equation}\label{taukd}
\tau_k\,d\Big(y_k+\frac{x}{\tau_k}\Big)=\tau_k\Big|y_k+\frac{x}{\tau_k}-\xi_k\Big|=\tau_k\sqrt{1+\sum_{1\le i\le n-1}\Big(\frac{\partial g}{\partial x_i}(\xi'_k)\Big)^2}\times\Big|y_{k,n}+\frac{x_n}{\tau_k}-\xi_{k,n}\Big|
\end{equation}
for all $k\in\N$ large enough. Now, since
$$\big|y_k-\xi_k|\le\Big|y_k+\frac{x}{\tau_k}-\xi_k\Big|+\frac{|x|}{\tau_k}=d\Big(y_k+\frac{x}{\tau_k}\Big)+\frac{|x|}{\tau_k}\le\frac{2|x|}{\tau_k}$$
by~\eqref{tauk}, it follows in particular that
$$|y'_k-\xi'_k|=O\Big(\frac{1}{\tau_k}\Big)\ \hbox{ as }k\to+\infty.$$
On the other hand, for $k\in\N$ large enough such that $y_k$, $\xi_k\in\partial\Omega\cap B_r$ and the segment $[y'_k,\xi'_k]$ is included in $V$, there holds
$$|y_{k,n}-\xi_{k,n}|=|g(y'_k)-g(\xi'_k)|\le\Big(\max_{[y'_k,\xi'_k]}|\nabla g|\Big)\times|y'_k-\xi'_k|.$$
Since $\lim_{k\to+\infty}y'_k=\lim_{k\to+\infty}\xi'_k=0$ and $g$ is (at least) of class $C^1$ in $V$ with $\nabla g(0)=0$, one infers that $|y_{k,n}-\xi_{k,n}|=o(1/\tau_k)$ as $k\to+\infty$. Finally, using~\eqref{taukd}, $\nabla g(0)=0$ and $x_n>0$, one concludes that
$$\tau_k\,d\Big(y_k+\frac{x}{\tau_k}\Big)=\sqrt{1+\sum_{1\le i\le n-1}\Big(\frac{\partial g}{\partial x_i}(\xi'_k)\Big)^2}\times\big|\tau_k(y_{k,n}-\xi_{k,n})+x_n\big|\to x_n\ \hbox{ as }k\to+\infty.$$
This is the desired result~\eqref{claim} and the proof of Lemma~\ref{convprofile} is thereby complete.
\end{proof}

\begin{proof}[Proof of Theorem~\ref{thnabla}] Fix any real number $M>0$. Let us first show that $\min_{\ol{\Omega^{M/\tau}}}|\nabla\pt|\gtrsim\tau$ as $\tau\to+\infty$. Assume not. Then there exist some sequences~$(\tau_k)_{k\in\N}$ of positive real numbers and $(x_k)_{k\in\N}$ in $\ol\O$ such that
\begin{equation}\label{xktauk}
x_k\in{\ol{\Omega^{M/\tau_k}}}\hbox{ for all }k\in\N,\ \ \tau_k\to+\infty\ \hbox{ and }\frac{|\nabla\varphi_{\tau_k}(x_k)|}{\tau_k}\to0\ \hbox{ as }k\to+\infty.
\end{equation}
For every $k\in\N$ large enough, there is a unique
\begin{equation}\label{defyk}
y_k\in\partial\O\ \hbox{ such that }d(x_k)=|x_k-y_k|.
\end{equation}
Consider now, for all $k\in\N$ large enough, the functions
\begin{equation}\label{defpsik}
\psi_k(x)=\varphi_{\tau_k}\Big(y_k+\frac{x}{\tau_k}\Big),
\end{equation}
which are defined in $\ol{\O_k}$ with
\begin{equation}\label{defomegak}
\O_k=\tau_k(\O-y_k).
\end{equation}
Up to extraction of a subsequence, one can assume that
\begin{equation}\label{defy}
y_k\to y\in\partial\Omega\ \hbox{ as }k\to+\infty.
\end{equation}
Denote $H$ the half-space
\begin{equation}\label{defH}
H=\big\{x\in\R^n:\ x\cdot\nu(y)<0\big\}.
\end{equation}
Since $\tau_k\to+\infty$ and $\nu(y_k)\to\nu(y)$ as $k\to+\infty$, it follows that, for any compact set $K\subset H$ (resp. $K\subset\R^n\setminus\overline{H}$), one has $K\subset\Omega_k$ (resp. $K\cap\overline{\O_k}=\emptyset$) for all $k$ large enough. Furthermore, the functions $\psi_k:\overline{\Omega_k}\to[0,1]$ satisfy
\begin{equation}\label{eqpsik}\left\{\begin{array}{ll}
\displaystyle-\Delta\psi_k-|\nabla\psi_k|=\frac{\lambda(\tau_k)}{\tau_k^2}\,\psi_k & \hbox{in }\Omega_k,\vspace{3pt}\\
\psi_k=0 & \hbox{on }\partial\O_k,\end{array}\right.
\end{equation}
with $\lambda(\tau_k)/\tau_k^2\to0$ as $k\to+\infty$. Since the families $(\partial\Omega_k\cap B_R)_{k\in\N}$ are bounded in~$C^2$ for every $R>0$, standard elliptic estimates up to the boundary and Sobolev injections imply that the sequences $(\|\psi_k\|_{W^{2,p}(\Omega_k\cap B_R)})_{k\in\N}$ and $(\|\psi_k\|_{C^{1,\alpha}(\overline{\Omega_k}\cap\overline{B_R})})_{k\in\N}$ are bounded for every $1\le p<+\infty$ and $0\le\alpha<1$. From the equations~\eqref{eqpsik} satisfied by the functions $\psi_k$, one also infers that, for any compact set $K\subset H$, the functions $\psi_k$ are also bounded in $C^{2,\alpha}(K)$ for $k$ large enough. Therefore, there is a function $\psi\in C^2(H)$ such that, up to extraction of a subsequence,
\begin{equation}\label{defpsi}
\psi_k\to\psi\ \hbox{ in }C^2_{loc}(H)\ \hbox{ as }k\to+\infty.
\end{equation}
From the previous observations, the function $\psi$ can also be extended as a $C^{1,\alpha}_{loc}(\overline{H})$ function (for all $0\le\alpha<1$) such that $\psi=0$ on $\partial H$ and
\begin{equation}\label{convpsi}
|\psi_k(\xi_k)-\psi(\xi)|+|\nabla\psi_k(\xi_k)-\nabla\psi(\xi)|\to0\ \hbox{ as }k\to+\infty
\end{equation}
for any sequence $(\xi_k)_{k\in\N}$ such that $\xi_k\in\overline{\Omega_k}$ for all $k\in\N$ and $\xi_k\to\xi\in\overline{H}$ as $k\to+\infty$. Consider now, for all $k\in\N$ large enough, the points
\begin{equation}\label{defzk}
z_k=\tau_k(x_k-y_k)\in\overline{\Omega_k}.
\end{equation}
Since $x_k\in\ol{\Omega^{M/\tau_k}}$ for all $k\in\N$ and $d(x_k)=|x_k-y_k|$, it follows that $|z_k|=\tau_kd(x_k)\le M$ for all $k\in\N$ large enough. Therefore, up to extraction of a subsequence, there is $z\in\overline{H}$ such that
\begin{equation}\label{defz}
z_k\to z\ \hbox{ as }k\to+\infty,
\end{equation}
hence $\nabla\psi_k(z_k)\to\nabla\psi(z)$ as $k\to+\infty$ by~\eqref{convpsi}. Since
$$\nabla\psi_k(z_k)=\frac{\nabla\varphi_{\tau_k}(x_k)}{\tau_k}\to0\ \hbox{ as }k\to+\infty$$
by~\eqref{xktauk} and~\eqref{defpsik}, it follows that $\nabla\psi(z)=0$. But, finally, Lemma~\ref{convprofile} implies that
$$\psi(x)=1-e^{x\cdot\nu(y)}$$
for all $x\in H$ and then for all $x\in\overline{H}$ by continuity. In particular, $|\nabla\psi(z)|=e^{z\cdot\nu(y)}>0$. This leads to a contradiction. As a consequence, the assumption~\eqref{xktauk} is ruled out, hence $\min_{\ol{\Omega^{M/\tau}}}|\nabla\pt|\gtrsim\tau$ as $\tau\to+\infty$.

Let us now show that $\max_{\ol{\O^{M/\tau}}}\big|\nabla\pt/|\nabla\pt|-\nabla d\big|\to0$ as $\tau\to+\infty$. Once again, argue by way of contradiction and assume that there exist $\epsilon>0$, a sequence $(\tau_k)_{k\in\N}$ of positive real numbers converging to $+\infty$, and a sequence $(x_k)_{k\in\N}$ in $\ol\O$ such that
\begin{equation}\label{xktaukbis}
x_k\in{\ol{\Omega^{M/\tau_k}}}\ \hbox{ and }\ \left|\frac{\nabla\varphi_{\tau_k}(x_k)}{|\nabla\varphi_{\tau_k}(x_k)|}-\nabla d(x_k)\right|\ge\epsilon>0\ \hbox{ for all }k\in\N.
\end{equation}
Up to extraction of a subsequence, let $y_k$, $\psi_k$, $\Omega_k$, $y$, $H$, $\psi$, $z_k$ and $z$ be defined as in~\eqref{defyk},~\eqref{defpsik},~\eqref{defomegak},~\eqref{defy},~\eqref{defH},~\eqref{defpsi},~\eqref{defzk} and~\eqref{defz}. On the one hand, as above there holds $\nabla\psi_k(z_k)\to\nabla\psi(z)=-e^{z\cdot\nu(y)}\nu(y)\neq0$ as $k\to+\infty$, hence 
\begin{equation}\label{nablapsik}
\frac{\nabla\psi_k(z_k)}{|\nabla\psi_k(z_k)|}\to-\nu(y)\ \hbox{ and }\ \frac{\nabla\varphi_{\tau_k}(x_k)}{|\nabla\varphi_{\tau_k}(x_k)|}\to-\nu(y)\ \ \hbox{ as }k\to+\infty.
\end{equation}
On the other hand, since $|x_k-y_k|=d(x_k)\to0$ and $y_k\to y\in\partial\Omega$, one infers that $\nabla d(x_k)\to\nabla d(y)=-\nu(y)$ as $k\to+\infty$. Together with~\eqref{nablapsik}, one gets a contradiction with~\eqref{xktaukbis}. Finally, $\max_{\ol{\O^{M/\tau}}}|\nabla\pt/|\nabla\pt|-\nabla d|\to0$ as $\tau\to+\infty$ and the proof of Theorem~\ref{thnabla} is thereby complete.
\end{proof}


\SE{The case of the annulus $\O=B_R\setminus \ol{B_r}$}\label{secradial}

This section is devoted to the proof of Theorem~\ref{thannulus}. It is actually an immediate consequence of the following proposition.

\begin{proposition}\label{pro:annulus}
	Let $\O$ be the annulus $\O=B_R\setminus \ol{B_r}$ with $0<r<R$ and $n\ge2$. Then, for every $\tau\ge0$, the \pf\ $\vp_\tau$ solving~\eqref{evm} is radially symmetric, i.e., $\pt(x)=\Phi_\tau(|x|)$ in $\ol\O$ for some function $\Phi_\tau:[r,R]\to[0,1]$. Furthermore, there exists $r_\tau\in(r,\frac{r+R}2)$ such that $\Phi_\tau'>0$ in $[r,r_\tau)$, $\Phi_\tau'<0$ in $(r_\tau,R]$, and
	$$r_\tau\to\frac{r+R}2\ \hbox{ as }\tau\to+\infty.$$
\end{proposition}

\begin{proof}
	For any $\tau\ge0$, the symmetry of $\pt$ follows from the uniqueness of the vector field $v_\tau$ minimizing~\eqref{lt} and the invariance of $\O$ by rotation. The function $\Phi_\tau:[r,R]\ni\rho\mapsto\Phi_\tau(\rho)\in[0,1]$ defined by~$\varphi_\tau(x)=\Phi_\tau(|x|)$ for every $x\in\ol\O$, is a $C^2([r,R])$ solution of the equation
	\Fi{radial}
	-\Phi_\tau''-\frac{n-1}{\rho}\Phi_\tau'-\tau|\Phi_\tau'|=\lt\Phi_\tau,\quad\rho\in[r,R],
	\Ff
	with $\Phi_\tau>0$ in $(r,R)$ and $\Phi_\tau(r)=\Phi_\tau(R)=0$. Since $\lt>0$ and $\Phi_\tau>0$ in $(r,R)$, we see that $\Phi_\tau''<0$ at all interior critical points of $\Phi_\tau$, which readily implies the existence of a unique radius
	$$r_\tau\in(r,R)$$
	at which $\Phi$ changes monotonicity. More precisely, one infers that $\Phi_\tau'>0$ in $(r,r_\tau)$ and $\Phi_\tau'<0$ in $(r_\tau,R)$. Hopf lemma also yields $\Phi_\tau'(r)>0$ and $\Phi_\tau'(R)<0$. Furthermore, Theorem~\ref{thmaxima} implies that
	$$r_\tau\to\frac{r+R}{2}\ \hbox{ as }\tau+\infty.$$
	As a consequence, for any given $x\in\overline{\O}$, there holds, for $\tau$ large enough,
	$$-\frac{v_\tau(x)}{\tau}=\frac{\nabla\pt(x)}{|\nabla\pt(x)|}=\left\{\begin{array}{rl}
	\displaystyle\frac{x}{|x|}=\nabla d(x) & \hbox{if }\displaystyle r\le |x|<\frac{r+R}{2},\vspace{3pt}\\
	-\displaystyle\frac{x}{|x|}=\nabla d(x) & \hbox{if }\displaystyle \frac{r+R}{2}<|x|\le R.\end{array}\right.$$
	These properties are actually sufficient to get the conclusion of Theorem~\ref{thannulus}.
	
	In order to complete the proof of Proposition~\ref{pro:annulus}, let us also show the additional property $r_\tau<(r+R)/2$. To do so, for any given $\tau\ge0$, consider the reflection of $\Phi$ with respect to $r_\tau$, i.e.,
	$$\t\Phi_\tau(\rho):=\begin{cases}
	\Phi_\tau(\rho) & \text{if }\rho\in[r,r_\tau]\vspace{3pt}\\
	\Phi_\tau(2r_\tau-\rho) & \text{if }\rho\in(r_\tau,2r_\tau-r].
	\end{cases}$$
	The function $\t\Phi_\tau$ is of class $C^2([r,2r_\tau-r])$ and it satisfies, for $\rho\in(r_\tau,2r_\tau-r]$,
	$$-\t\Phi_\tau''-\frac{n-1}{\rho}\t\Phi_\tau'-\tau|\t\Phi_\tau'|-\lt\t\Phi_\tau=\left(\frac{n-1}{\rho}+\frac{n-1}{2r_\tau-\rho}\right)\Phi_\tau'(2r_\tau-\rho)>0.$$ 
	Namely, $\t\Phi_\tau$ is a strict supersolution of~\eqref{radial} in $(r_\tau,2r_\tau-r]$, that is, the function $\t\vp_\tau\in C^2(\overline{B_{2r_\tau-r}}\setminus B_r)$ defined by
	$$\t\vp_\tau(x):=\t\Phi_\tau(|x|),\ \ x\in\overline{B_{2r_\tau-r}}\setminus B_r,$$
	satisfies
	$$-\Delta\t\vp_\tau-\tau|\nabla\t\vp_\tau|>\lt\t\vp_\tau\ \hbox{ in }\overline{B_{2r_\tau-r}}\setminus\ol{B_{r_\tau}}.$$
	It also satisfies $-\Delta\t\vp_\tau-\tau|\nabla\t\vp_\tau|\ge\lt\t\vp_\tau$ in $\ol{B_{r_\tau}}\setminus B_r$, where it coincides with $\pt$. Thus, calling $\t\lambda(\tau)$ the eigenvalue given by problem~\eqref{evm} with domain $B_{2r_\tau-r}\setminus\ol{B_r}$ and observing that $\t\vp_\tau>0$ in $B_{2r_\tau-r}\setminus\ol{B_r}$, the first characterization in formula~\eqref{char} of Proposition~\ref{pro:char} yields
	$$\t\lambda(\tau)\geq\lt.$$
	Then, by the monotonicity of $\lt$ provided by Proposition~\ref{pro:char}, we infer that $B_{2r_\tau-r}\setminus\ol{B_r}\subset B_R\setminus\ol{B_r}$. Moreover, the inclusion is strict because otherwise, again by Proposition~\ref{pro:char}, $\t\vp_\tau$ would coincide with $\pt$ up to a scalar multiple, but we know that $\t\vp_\tau$ is a strict supersolution of the equation satisfied by $\pt$ in $\ol{B_{2r_\tau-r}}\setminus\ol{B_{r_\tau}}$. This shows that $2r_\tau-r<R$, that is, $r_\tau<(r+R)/2$. The proof of Proposition~\ref{pro:annulus} is thereby complete.
\end{proof}



\begin{thebibliography}{1}

\bibitem{BNV} Berestycki, H.; Nirenberg, L.; Varadhan, S.R.S. The principal eigenvalue and maximum principle for second order elliptic operators in general domains. \emph{Comm. Pure Appl. Math.} \textbf{47} (1994), 47--92.
\bibitem{CanSin} Cannarsa, P.; Sinestrari, C. \emph{Semiconcave functions, {H}amilton-{J}acobi equations, and optimal control}, \emph{Progress in Nonlinear Differential Equations and their Applications}, vol.~58, Birkh\"auser Boston, Inc., Boston, MA, 2004.
\bibitem{Faber} Faber, G. Beweis, dass unter allen homogenen Membranen von gleicher Fl\"ache und gleicher Spannung die kreisf\"ormige den tiefsten Grundton gibt, \emph{Sitzungsberichte der mathematisch-physikalischen Klasse der Bauerischen Akademie der Wissenschaften zu M\"unchen} (1923), 169--172.
\bibitem{friedman} Friedman, A. The asymptotic behavior of the first real eigenvalue of a second order elliptic operator with a small parameter in the highest derivatives. \emph{Indiana Univ. Math. J.} \textbf{22} (1972/1973), 1005--1015.
\bibitem{GT} Gilbarg, D.; Trudinger, N.S. \emph{Elliptic partial differential equations of second order}, Springer Verlag, Berlin, 1983.
\bibitem{HNRcras} Hamel, F.; Nadirashvili, N.; Russ, E. An isoperimetric inequality for the principal eigenvalue of the Laplacian with drift. \emph{C.~R.~Acad. Sci. Paris Ser.~I} \textbf{340} (2005), 347--352.
\bibitem{HNRpreprint} Hamel, F.; Nadirashvili, N.; Russ, E. A Faber-Krahn inequality with drift (2006), \url{https://arxiv.org/abs/math/0607585}. 
\bibitem{HNR} Hamel, F.; Nadirashvili, N.; Russ, E. Rearrangement inequalities and applications to isoperimetric problems for eigenvalues. \emph{Ann. of Math.	(2)} \textbf{174} (2011),  647--755.
\bibitem{Krahndim2} Krahn, E. \"Uber eine von Rayleigh formulierte Minimaleigenschaft des Kreises. \emph{Math. Ann.} \textbf{94} (1925), 97--100.
\bibitem{Krahndimn} Krahn, E. \"Uber Minimaleigenschaft der Kugel in drei und mehr Dimensionen, \emph{Acta Comm. Univ. Tartu (Dorpat)} \textbf{A9} (1926), 1--44.

\end{thebibliography}
\end{document}